\documentclass[12pt]{article}
\usepackage{arxiv}
\usepackage{amsmath,amssymb,amsthm,graphicx}
\usepackage{caption}
\usepackage{abstract}
\usepackage{indentfirst}
\usepackage{lipsum}
\usepackage{ragged2e}
\usepackage{rotating}
\usepackage{hyperref}
\usepackage{varioref}
\usepackage[rawfloats=true]{floatrow} 
\usepackage{anysize}                
\usepackage{setspace}               
\usepackage{parskip}
\usepackage{booktabs}
\usepackage{pdflscape}
\usepackage{float}
\usepackage{cleveref}
\usepackage{setspace}
\usepackage{subcaption} 
\usepackage{graphicx}
\usepackage{enumitem}

\theoremstyle{plain}
\newtheorem{thm}{Theorem}[section]

\newtheorem{prop}[thm]{Proposition}

\newtheorem{definition}{Definition}[section] 
\date{}
\usepackage{doi}

\title{Solving Fuzzy Quadratic Programming Problems with a Proposed Algorithm}
\author{ \href{https://orcid.org/0000-0003-1784-2741}{\includegraphics[scale=0.06]{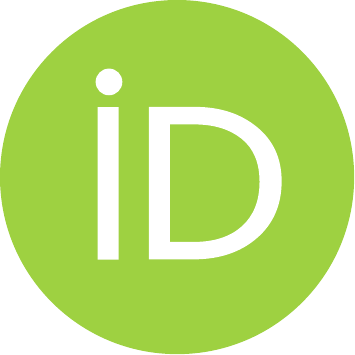}\hspace{1mm}Sajal Chakroborty}\thanks{Other affiliation : East West University, Dhaka, Bangladesh, email: sajal.math@yahoo.com} \\
	Department of Mathematics and Statistics\\
	Texas Tech University\\
	Lubbock, Texas, USA \\
	\texttt{email:sajal.chakroborty@ttu.edu} \\}
\begin{document}
\maketitle
\begin{abstract}
 The theory of fuzzy mathematics has been proven very effective for defining and solving optimization problems. Fuzzy quadratic programming (FQP) is a consequence of this approach. In this paper, an algorithm has been proposed to solve FQP with coefficients as triangular fuzzy numbers (TFN). The proposed algorithm converts FQP into two parametric quadratic programming (QP) problems. These QP solutions provide a lower and upper bound on the objective function of FQP. When these two values coincide, an optimal solution is achieved. This algorithm has been analyzed using a numerical example and compared with existing methods.   
\end{abstract}
\section{Introduction}
Quadratic programming (QP) is an optimization technique and a generalization of linear programming (LP). It's theoretical foundation was developed by M. Frank and P. Wolfe \cite{frank1956algorithm}. QP is widely used in regression analysis, portfolio optimization, image and signal processing, etc. It is a nonlinear optimization problem that has a quadratic objective function with linear constraints. It can be formulated in a general form as \cite{liberman1988introduction},
	\begin{equation} \label{eqn: QDTC 1.1}
		\begin{split}
			\text{min}\; z=\sum_{j=1}^{n} c_{j}x_{j}+\frac{1}{2} \sum_{i=1}^{n} \sum_{j=1}^{n}q_{ij}x_{i}x_{j}\\       
		\text{s.t.} \; \; \; \; \;	 \sum_{j=1}^{n} a_{ij}x_{j} \leq b_{i}, \; i=1,\dots,n \\
			x_{j} \geq 0, \; j=1,\dots,n
		\end{split}
	\end{equation}
where $x_{i}$ are decision variables to be determined, $c_i$ represent cost coefficients, and $q_{ij}$ are the coefficients of quadratic form for $i,j=1,2,\dots,n$. Note that the quadratic matrix $Q= \left[q_{ij} \right]$ is symmetric and positive definite. The constraint coefficients and right hand constants are denoted by $a_{ij}$ and $b_i$ respectively.  \\
The above formulation does not take into account the uncertain and vague data. But it can be dealt with using fuzzy sets and logic \cite{zadeh1996fuzzy}. For this reason, fuzzy optimization techniques have become very useful. Zimmermann first applied this theory to mathematical programming \cite{zimmermann1978fuzzy}. Recently, many ideas have been proposed to formulate and solve QP problems in a fuzzy environment. For instance, Shiang-Tai Liu formulated QP as an FQP where all the coefficients are triangular fuzzy numbers (TFN) \cite{liu2009revisit}. Liu has developed an algorithm where he proposed to create a pair of two-level mathematical programs to calculate the upper and lower bound of the objective function of FQP \cite{liu2007solving, liu2009revisit}. Then the optimal solution was derived by solving two QP. Amin Mansoori, Sohrab Effati, and Mohammad Eshangnezhad designed a one-layer structured neural network model for solving the FQP \cite{mansoori2018neural}. They changed the FQP to a bi-objective problem, then reduced it to a weighting problem, and finally constructed a Lagrangian dual. They have solved their proposed dynamical system by ODE method. To plan a waste management system under uncertainty, Y. P. Li and G. H. Huang developed fuzzy two-stage quadratic programming (FTSQP) \cite{li2007fuzzy}. Q. Tan, G. H. Huang, and Y. P. Cai developed a superiority-inferiority-based inexact fuzzy stochastic quadratic programming (SI-IFSQP) model for sustainable water supply under multiple uncertainties \cite{tan2013multi}.\\
Motivated by the existing and recent developments in fuzzy optimization, a new algorithm has been developed in this paper to solve FQP. This paper has been organized as,
in section 2, the definitions and properties of fuzzy sets have been discussed briefly. In section 3, a new approach has been proposed to solve FQP by converting it into two QP problems. The proposed method will be demonstrated using a numerical example in the next section. Then a comparison has been presented with existing methods. Finally, a conclusion has been made.

\section{Preliminaries}
A few definitions and properties of fuzzy set theory and optimization are briefly explored in this section \cite{wang2003directional,wu2004evaluate, zimmermann1978fuzzy}.  
	\begin{definition} 
	Let $\mathcal{X}$ be a universal set. Then a fuzzy set $\tilde{A}$ in $\mathcal{X}$ with the membership function $\mu_{\tilde{A}} : \mathbb{R} \rightarrow [0,1]$ is defined as, 
		\begin{equation*}\label{eqn: FZY 1.1}
			\tilde{A}=\{(x,\mu_{\tilde{A}}(x)) : x \in \mathcal{X}\}
		\end{equation*}
	Support of $\tilde{A}$ is a crisp set $C(\tilde{A})=\{x \in \mathcal{X}: \mu_{\tilde{A}}(x)>0\}$. 
	\end{definition}
    \begin{definition} 
    	The $\alpha-$level set of $\tilde{A}$ is a crisp set $A_{\alpha}$, defined as
    	\begin{equation*}
    		A_{\alpha}=\{x\in \mathcal{X}: \mu_{\tilde{A}}(x) \geq \alpha\}, \alpha \in \left[0,1\right]
    	\end{equation*}
    \end{definition}
    \begin{definition} 
    	$\tilde{A}$ is called convex if $\mu_{\tilde{A}} \left(\lambda x_1 + (1-\lambda) x_2\right) \geq min \{\mu_{\tilde{A}}(x_1), \mu_{\tilde{A}}(x_2)\}$, where $x_1, x_2 \in \mathcal{X}$, $\lambda \in \left[0,1\right]$.
    \end{definition}
\begin{definition} 
   If $\mu_{\tilde{A}}=1$ for at least one $x\in \mathcal{X}$ then $\tilde{A}$ is called normal.
\end{definition}
\begin{definition}  
    If $\tilde{A}$ is a normal, and convex fuzzy set with a bounded support then it is called a fuzzy number (FN). 
\end{definition}
\begin{definition} 
	The Triangular fuzzy number (TFN) $\tilde{A}=\left(a_{1}, a_{2}, a_{3}\right)$ is a FN with the membership function defined as
	\begin{equation}
		\mu_{\tilde{A}}(x)=\begin{cases} 
			\frac{x-a_{2}}{a_{2}-a_{1}}, & x \in \left[a_{1},a_{2}\right] \\
			\frac{a_{3}-x}{a_{3}-a_{2}}, & x \in \left[a_{2}, a_{3}\right] \\
			0, & x<a_{1} \; and \; x>a_{3}
		\end{cases}
	\end{equation}
where $a_{1}, a_{2}, a_{3} \in \mathbb{R}$ and $a_{1} \leq a_{2} \leq a_{3}$. 
\end{definition}
Note that a FN can be fully and uniquely represented by its $\alpha-$cut. The $\alpha-$cut for a TFN is $\tilde{A}_{\alpha}=\left[A_{\alpha}^{L}, A_{\alpha}^{U}\right]$, $\forall \alpha \in \left[0,1\right]$, where $A_{\alpha}^{L}= a_{1}+\alpha(a_{2}-a_{1})$, and $A_{\alpha}^{U}=a_{3}-\alpha(a_{3}-a_{2})$. This representation allow us to perform arithmetic operation between two FN $\tilde{A}$, and $\tilde{B}$ as,    
\begin{enumerate}[label=\roman*]
	\item Addition: $(\tilde{A}+\tilde{B})_{\alpha}=[A^{L}_{\alpha}+B^{L}_{\alpha}, A^{U}_{\alpha}+B^{U}_{\alpha}]$
	\item Scalar Multiplication: If $k>0$ then $(k\tilde{A})_{\alpha}=[kA^{L}_{\alpha}, k A^{U}_{\alpha}]$ and  $(k\tilde{A})_{\alpha}=[k A^{U}_{\alpha}, kA^{L}_{\alpha}]$ if $k<0$. 
	\item Multiplication: $(\tilde{A}.\tilde{B})_{\alpha}= \left[M_{\min},M_{\max}\right]$, where $M_{min}=\min(A^{L}_{\alpha} B^{L}_{\alpha},A^{L}_{\alpha} B^{U}_{\alpha}, A^{U}_{\alpha}B^{L}_{\alpha}, A^{U}_{\alpha}B^{U}_{\alpha})$ \\ and $M_{\max}= \max(A^{L}_{\alpha} B^{L}_{\alpha},A^{L}_{\alpha} B^{U}_{\alpha}, A^{U}_{\alpha}B^{L}_{\alpha}, A^{U}_{\alpha}B^{U}_{\alpha})$. 
\end{enumerate}
A subset $A$ of $X$ is said to be convex, if $\lambda x+ \left(1-\lambda \right)y \in A$ whenever $x,y \in A$ and $\lambda \in \left(0,1 \right)$. Let $E$ denotes a collection of TFN. Then, 
\begin{definition}
    A fuzzy mapping $\tilde{f}: A \rightarrow E$ defined on a convex susbset $A$ in $\mathcal{X}$ is convex, if and only if, 
    \begin{equation*}
        \tilde{f} \left(\lambda x+\left(1-\lambda \right)y \right) \leq \lambda \tilde{f}(x) + \left (1-\lambda \right) \tilde{f}(y), \forall x, y \in A, \lambda \in \left[0,1\right]
    \end{equation*}
If $\tilde{f} \left(\lambda x+\left(1-\lambda \right)y \right) < \lambda \tilde{f}(x) + \left (1-\lambda \right) \tilde{f}(y), \forall x, y \in A, x \neq y, \lambda \in \left(0,1\right)$, then $\tilde{f}$ is called strictly convex.  
\end{definition}
Like the $\alpha$-cut of a fuzzy number, a fuzzy mapping $\tilde{f}$ can be written as follow, 
\begin{equation*}
    \tilde{f}(x)= \left[f^{L}_{\alpha}, f^{U}_{\alpha}\right], \forall \alpha \in \left[0,1\right]
\end{equation*}
Now a fuzzy optimization problem can be defined as,
\begin{definition}
Let $\Omega \subset \mathbb{R}^{n}$ be a convex set which contains all feasible solutions, then
    \begin{equation*}
              \begin{split}
                \min \; \tilde{f}(x) & = \left[f^{L}_{\alpha}, f^{U}_{\alpha}\right]\\
             \text{s.t.} \; \; x & \in \Omega
        \end{split}  
    \end{equation*}
\end{definition}
In the following section, FQP has been formulated with fuzzy coefficients as TFN.
\section{Fuzzy Quadratic Programming }
FQP is similar to QP in terms of formulation. It can be formulated as, 
\begin{equation} \label{eqn: QDTC 1.2}
	\begin{split}
		\min \; \tilde{z}=\sum_{j=1}^{n} \tilde{c}_{j}x_{j}+\frac{1}{2} \sum_{i=1}^{n} \sum_{j=1}^{n} \tilde{q}_{ij}x_{i}x_{j}\\
		\text{s.t.} \; \; \; \; \; \sum_{j=1}^{n} \tilde{a}_{ij}x_{j} \leq \tilde{b}_{i}, i=1,\dots,n\\
		x_{j} \geq 0, \; j=1,\dots,n
	\end{split}	
\end{equation}
where the coefficients $\tilde{c},\tilde{q}_{ij}, \tilde{a}_{ij}$ and $\tilde{b}$ are triangular fuzzy number (TFN). For $i,j=1,2,\dots,n$ let $\tilde{c}_{j}=\left(c^{1}_{j},c^{2}_{j},c^{3}_{j}\right)$, $\tilde{q}_{ij}= \left(q^{1}_{ij}, q^{2}_{ij}, q^{3}_{ij}\right)$, $\tilde{a}_{ij}= \left(a^{1}_{ij}, a^{2}_{ij}, a^{3}_{ij}\right)$, and $\tilde{b}_{i}= \left(b^{1}_{i}, b^{2}_{i}, b^{3}_{i}\right)$. In vector matrix notation (\ref{eqn: QDTC 1.2}) can be simplified as,
\begin{equation} \label{eqn: FQDTC 1.1}
    \begin{split}
        \min \; \tilde{Z}  = \tilde{C}^{T}X & +\frac{1}{2} X^{T} \tilde{Q}X \\
         \text{s.t.} \;\; \tilde{A}X & \leq \tilde{b}\\
        X & \geq 0
    \end{split}
\end{equation}
\begin{prop}
If $\tilde{Q}$ is symmetric and positive semi-definite, then $\tilde{f}(x)=\tilde{C}^{T}X+\frac{1}{2}X^{T} \tilde{Q}X$ in (\ref{eqn: FQDTC 1.1}) is a convex fuzzy mapping. 
\end{prop}
\begin{proof}
For any $\lambda \in \left[0,1\right]$, $\lambda x + (1-\lambda)y \in \Omega$, implies,
\begin{equation*}
    \begin{split}
     \tilde{f} \left(x \lambda+(1-\lambda)y\right) &= \tilde{C}^{T} (x \lambda+(1-\lambda)y)+\frac{1}{2} \left(\lambda x+ (1-\lambda)y \right)^{T} \tilde{Q} \left(\lambda x+ (1-\lambda) y\right)\\
     &= \tilde{C}^{T} (x \lambda+(1-\lambda)y)+\frac{1}{2} \left(y+ \lambda (x-y) \right)^{T} \tilde{Q} \left(y+ \lambda (x-y) \right)\\
     &= \tilde{C}^{T} (x \lambda+(1-\lambda)y)+\frac{1}{2} \left(y^{T}\tilde{Q}y+ \lambda^{2} (x-y)^{T} \tilde{Q}(x-y)+2 \lambda y^{T} \tilde{Q} (x-y) \right)\\
     & \leq \tilde{C}^{T} (x \lambda+(1-\lambda)y)+\frac{1}{2} \left(y^{T}\tilde{Q}y+ \lambda (x-y)^{T} \tilde{Q}(x-y)+2 \lambda y^{T} \tilde{Q} (x-y) \right)\\
     &=\tilde{C}^{T} (x \lambda+(1-\lambda)y)+\frac{1}{2} \left(y^{T}\tilde{Q}y+ \lambda (x+y)^{T} \tilde{Q}(x-y)\right)\\
     &= \tilde{C}^{T} (x \lambda+(1-\lambda)y)+\frac{1}{2} \left(\lambda x^{T} \tilde{Q}x+(1-\lambda) y^{T} \tilde{Q}y \right)\\
     &= \lambda \left(\tilde{C}^{T}x+\frac{1}{2}X^{T}\tilde{Q}x \right)+ \left(1-\lambda\right) \left(\tilde{C}^{T}y+\frac{1}{2}y^{T}\tilde{Q}y \right)\\
     &=\lambda \tilde{f}(x)+\left(1-\lambda \right) \tilde{f}(y)
    \end{split}
\end{equation*}
\end{proof}
Since the objective function of equation (\ref{eqn: FQDTC 1.1}) is a convex fuzzy mapping and $\{X: X \geq 0, \tilde{A}X \leq \tilde{b}\}$ is a convex feasible set, so (\ref{eqn: FQDTC 1.1}) is a convex fuzzy programming. Now if $X^{*}$ is a local optimal solution of (\ref{eqn: FQDTC 1.1}), then it will be a global solution of (\ref{eqn: FQDTC 1.1}) \cite{mansoori2018neural}. Also, if the objective function is strictly convex, the $X^{*}$ is the unique global optimal solution of (\ref{eqn: FQDTC 1.1}) according to \cite{mansoori2018neural}. 
In the next section, the proposed method has been discussed. 
\section{Proposed Method} 
First, define the real-life problem as (\ref{eqn: QDTC 1.2}). The arithmetic operations that have been used here are defined in section 2. Continue with the following steps.

\textbf{Step-1 :} To determine the $\alpha$-optimal value by taking the $\alpha$-cut of the objective function and the constraints (\ref{eqn: QDTC 1.2}) can be written as,
\begin{equation}\label{eqn: QDTC 1.3}
	\begin{split}
		\min \; \left[z^{L}_{\alpha}, z^{U}_{\alpha } \right]= \sum_{j=1}^{n} \left[c^{L}_{\alpha,j}, c^{U}_{\alpha,j}\right]x_{j}+ \frac{1}{2} \sum_{i=1}^{n} \sum_{j=1}^{n} \left[q^{L}_{\alpha,ij}, q^{U}_{\alpha,ij}\right]x_{i}x_{j} \\
		\text{s.t.} \; \; \; \; \; \;  \sum_{j=1}^{n} \left[a^{L}_{\alpha,ij}, a^{U}_{\alpha,ij}\right] x_{j}  \leq [b^{L}_{\alpha,i}, b^{U}_{\alpha,i}], \; i=1,\dots,n\\
		x_{j} \geq 0, \; \; j=1, \dots,n
	\end{split}
\end{equation}
\textbf{Step-2 :} Formulate the following QP for different $\alpha$ values to obtain the lower bound $z^L_{\alpha}$ of the objective function.
\begin{equation}\label{eqn: ALG 1.1}
	\begin{split}
		\min z^{L}_{\alpha} =\sum_{j=1}^{n} c^{L}_{\alpha,j}x_{j}+\frac{1}{2} \sum_{i=1}^{n} \sum_{j=1}^{n} q^{L}_{\alpha,ij} x_{i} x_{j} \\
		\text{s.t.} \; \; \; \; \sum_{j=1}^{n} a^{L}_{\alpha, ij} x_{j} \leq b^{L}_{\alpha, i}, \; \; i=1,\dots,n \\
		x_{j} \geq 0, \; \; j=1,\dots,n
	\end{split}
\end{equation}
\textbf{Step-3 :} Similarly, to find the upper bound $z^{U}_{\alpha}$ of $\tilde{z}$ for different $\alpha$  values formulate the following QP. 
\begin{equation} \label{eqn: ALG 1.2}
	\begin{split}
		\min z^{U}_{\alpha } =\sum_{j=1}^{n} c^{U}_{\alpha,j}x_{j}+\frac{1}{2} \sum_{i=1}^{n} \sum_{j=1}^{n} q^{U}_{\alpha,ij} x_{i} x_{j} \\
		\text{s.t.} \; \; \; \; \sum_{j=1}^{n} a^{U}_{\alpha, ij} x_{j} \leq b^{U}_{\alpha, i}, \; \; i=1,\dots,n \\
		x_{j} \geq 0, \; \; j=1,\dots,n
	\end{split}
\end{equation}
\textbf{Step-4 :} Choose an initial value of $x_0$ for the decision variables and solve the problem (\ref {eqn: ALG 1.1}) for a fixed $\alpha \in \left [0,1 \right]$.
\begin{equation*}
	\begin{split}
		y^{L}_{k}&=x_{k}-\frac{1}{K_{\alpha}} \nabla z^{L}_{\alpha}\\
		x_{k+1} &=proj_{\mathbb{S}_{L}} \left(y^{L}_{k}\right)
	\end{split}
\end{equation*}
 For fixed $\alpha$, $K_{\alpha}$ is the Lipschitz constant of $z^{L}_{\alpha}$ and $proj_{\mathbb{S}_{L}}$ is the projection on the half space $\mathbb{S}_{L}= \{x: A^{L}_{\alpha}x \leq b^{L}_{\alpha}\}$.

\textbf{Step-5 :} Follow the same procedure as step-4 and perform the following steps for (\ref{eqn: ALG 1.2}). 
\begin{equation*}
	\begin{split}
		y^{U}_{k}&=x_{k}-\frac{1}{M_{\alpha}} \nabla z^{U}_{\alpha}\\
		x_{k+1} &=proj_{\mathbb{S}_{U}} \left(y^{U}_{k}\right)
	\end{split}
\end{equation*}
Like previous step, $M_{\alpha}$ is the Lipschitz constant of $z^{U}_{\alpha}$ and $proj_{\mathbb{S}_{U}}$ is the projection on the half space $\mathbb{S}_{U}= \{x: A^{U}_{\alpha}x \leq b^{U}_{\alpha}\}$.

\textbf{Step-6 :} If $z^{L}_{\alpha}=z^{U}_{\alpha} $ for any $\alpha$, then stop. 
\section{Numerical Example}
In this section, effectiveness of the proposed algorithm has been analyzed with an example.
Consider the following FQP from \cite{liu2009revisit}
\begin{equation*}\label{eqn: EXAMPLE 1.1}
\begin{split}
			\min \; \tilde{z}  = \left(-6,-5,-4\right)x_{1} + \left(1,1.5,2\right)x_{2} +\frac{1}{2} \left[ \left(4,6,8\right)x_{1}^{2}+\left(-6,-4,-2\right)x_{1}x_{2}+\left(2,4,6\right)x_{2}^{2}\right]\\
\end{split}
\end{equation*}
\begin{equation*}
\begin{split}
		 \text{s.t.} \; \; \; \; \; \; \; x_{1}+\left(0.5,1,1.5\right)x_{2}  \leq \left(1,2,3\right) \\
		\;\; \left(1,2,3\right)x_{1}+\left(-2,-1,-0.5\right)x_{2} & \leq \left(3,4,5\right)\\
		 x_{1}, x_{2} \geq 0 &
\end{split}
\end{equation*}
At first, two QP models will be formulated according to Step-$3$ and $4$. For any $\alpha\in \left[0,1 \right]$ the lower bound $z^{L}_{\alpha}$ of $\tilde{z}$ will be obtained by solving,  
 \begin{equation}\label{eqn: EXAMPLE 1.2}
 	\begin{split}
 \min\; \; z^{L}_{\alpha}= \left(-6+\alpha\right)x_{1}+\left(1+0.5\alpha\right)x_2+ \frac{1}{2} & \left(x_1\; x_2\right)  \left( \begin{array}{cc}
 	4+2\alpha & -3+\alpha \\
 	-3+\alpha & 2+2\alpha
 \end{array} \right)
 \left( \begin{array}{cc}
 	x_1 \\
 	x_2
 \end{array} \right)\\
 		\text{s.t.} \;
		\left( \begin{array}{cc}
			1 & 0.5+0.5\alpha \\
			1+\alpha & -2+\alpha
		\end{array} \right)
		\left( \begin{array}{cc}
			x_1 \\
			x_2
		\end{array} \right) & \leq 
		\left( \begin{array}{cc}
			1+\alpha \\
			2+2\alpha
		\end{array} \right)\\
		x_{1}, x_{2} & \geq 0 
\end{split}
 \end{equation}
Similarly, to get the upper bound of $\tilde{Z}$, solve the following QP.  
 \begin{equation}\label{eqn: EXAMPLE 1.3}
	\begin{split}
		\min \;\; z^{U}_{\alpha} = \left(-4-\alpha\right)x_{1}+\left(2-0.5\alpha\right)x_2+& \frac{1}{2} \left(x_1\; x_2\right)  \left( \begin{array}{cc}
			8-2\alpha & -1-\alpha \\
			-1-\alpha & 6-2\alpha
		\end{array} \right)
		\left( \begin{array}{cc}
			x_1 \\
			x_2
		\end{array} \right)\\
	\text{s.t.} \; \;
		\left( \begin{array}{cc}
			1 & 1.5-0.5\alpha \\
			3-\alpha & 0.5-0.5\alpha
		\end{array} \right) &
		\left( \begin{array}{cc}
			x_1 \\
			x_2
		\end{array} \right)  \leq 
		\left( \begin{array}{cc}
			3-\alpha \\
		6-2\alpha
		\end{array} \right)\\
		x_{1}, x_{2} & \geq 0 
	\end{split}
\end{equation}
The detailed results with comparison to existing methods have been discussed in the following section.
\section{Results and Discussion}
For different values of $\alpha$, the objective function values of (\ref{eqn: EXAMPLE 1.2}) and (\ref{eqn: EXAMPLE 1.3}) have been presented in Table 1. Solutions obtained by \cite{liu2009revisit} and \cite{mansoori2018neural} are also presented in the same table with computation time. In Figure 1, the membership function of $\tilde{z}$ obtained by the proposed method has been presented along with Liu and Mansoori's method. In Figure 2, a comparison between the proposed method and the existing methods is shown.
\begin{table}[H]
\onehalfspacing
\centering{\captionof{table}{Objective values for different $\alpha$}}
	\begin{tabular}{@{}cccccccc@{}}
		\toprule
		$\alpha$                                                                                                                         & 0.0     & 0.2     & 0.4     & 0.6     & 0.8     & 1.0     & \begin{tabular}[c]{@{}c@{}}CPU\\ time (s)\end{tabular}\\ \midrule
		\begin{tabular}[c]{@{}c@{}}$z^{L}_{\alpha}$ in  Proposed \\ Method\end{tabular}                                                  & -4.0833 & -4.0503 & -3.7271 & -3.1306 & -2.4891 & -2.0872 & 0.032       \\
		
		\begin{tabular}[c]{@{}c@{}}$z^{U}_{\alpha}$ in Proposed\\  Method\end{tabular}                                                   & -1      & -1.1605 & -1.3444 & -1.5559 & -1.8    & -2.0872 & 0.029      \\
		
		\begin{tabular}[c]{@{}c@{}}$z^{L}_{\alpha}$ in Liu's \cite{liu2009revisit}\\ Method\end{tabular}                                                      & -10.08  & -6.72   & -4.46   & -3.14   & -2.49   & -2.09   & 0.050        \\
		
		\begin{tabular}[c]{@{}c@{}}$z^{U}_{\alpha}$ in Liu's \cite{liu2009revisit}\\  Method\end{tabular}                                                      & -1      & -1.16   & -1.34   & -1.45   & -1.80   & -2.09   & 0.055        \\
		
		\begin{tabular}[c]{@{}c@{}}$z^{L}_{\alpha}$ in Mansoori's\\ Method \cite{mansoori2018neural} with \\ $w_1=\frac{1}{4}$, $w_{2}=\frac{3}{4}$\end{tabular} & -1.4464 & -1.559  & -1.6735 & -1.8    & -1.9363 & -2.0875 & 0.033        \\
		
		\begin{tabular}[c]{@{}c@{}}$z^{U}_{\alpha}$ in Mansoori's\\ Method \cite{mansoori2018neural}with\\  $w_1=\frac{1}{3}$, $w_2=\frac{2}{3}$\end{tabular}     & -1.6333 & -1.7146 & -1.8    & -1.8997 & -1.9841 & -2.0875 & 0.034        \\ \bottomrule
	\end{tabular}
\end{table}
From the above table, it is clear that the sequence of objective values obtained by the proposed algorithm has converged at the same point as the existing methods. It generates more accurate solutions than others. In the following figure, the convergence of objective values for different methods is shown in terms of the triangular membership function. 
\begin{figure}[H] 
	\centering 
		\includegraphics[scale=0.6]{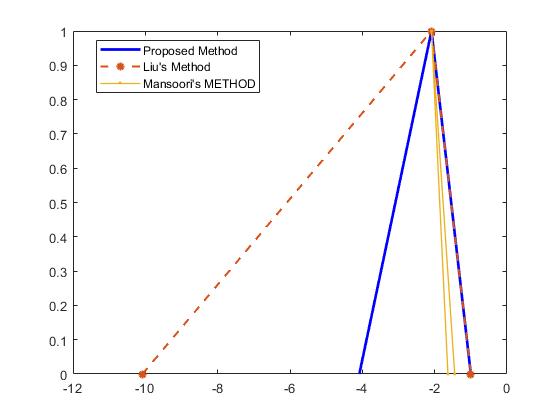} 
		\caption{Membership function for the Proposed Method vs Existing Methods} 
\end{figure} 
The gap between the lower and upper bounds of the objective values of Liu's method is larger than that of the proposed method. Although this gap is larger than Mansoori's method, anyone can make an intuition that the proposed method is slower than Mansoori's method. But the computation time tells another story. Though it could be easily observed from Table 1, let's look at the following figure.
\begin{figure}[H] 
	\centering 
	\includegraphics[scale=0.6]{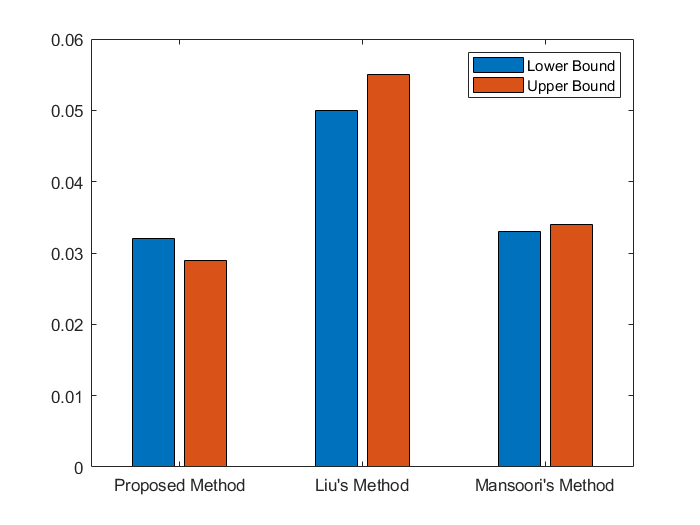} 
	\caption{Computational Time Comparison} 	
\end{figure} 
The proposed algorithm has taken 0.0305 seconds on average to solve the problem. On the other hand, Mansoori's method took 0.0335 seconds and Liu's method took 0.0525 seconds on average.
\section{Conclusion}
Solving FQP is more challenging than QP. In this paper, a new approach has been proposed to solve FQP. This approach has generated more accurate results than many existing methods while using less computation time. This idea can be easily extended to solve multi-objective FQP, fuzzy regression analysis, etc., which is the future goal of this work.

\bibliography{references.bib}

\begin{thebibliography}{10}

\bibitem{frank1956algorithm}
Marguerite Frank and Philip Wolfe.
\newblock An algorithm for quadratic programming.
\newblock {\em Naval research logistics quarterly}, 3(1-2):95--110, 1956.

\bibitem{li2007fuzzy}
YP~Li and Guo~H Huang.
\newblock Fuzzy two-stage quadratic programming for planning solid waste
  management under uncertainty.
\newblock {\em International Journal of Systems Science}, 38(3):219--233, 2007.

\bibitem{liberman1988introduction}
Hiller Liberman.
\newblock {\em Introduction to operations research}.
\newblock Libros McGraw-Hill de Mexico, 1988.

\bibitem{liu2007solving}
Shiang-Tai Liu.
\newblock Solving quadratic programming with fuzzy parameters based on
  extension principle.
\newblock In {\em 2007 IEEE International Fuzzy Systems Conference}, pages
  1--5. IEEE, 2007.

\bibitem{liu2009revisit}
Shiang-Tai Liu.
\newblock A revisit to quadratic programming with fuzzy parameters.
\newblock {\em Chaos, Solitons \& Fractals}, 41(3):1401--1407, 2009.

\bibitem{mansoori2018neural}
Amin Mansoori, Sohrab Effati, and Mohammad Eshaghnezhad.
\newblock A neural network to solve quadratic programming problems with fuzzy
  parameters.
\newblock {\em Fuzzy Optimization and Decision Making}, 17(1):75--101, 2018.

\bibitem{tan2013multi}
Qian Tan, GH~Huang, and YP~Cai.
\newblock Multi-source multi-sector sustainable water supply under multiple
  uncertainties: an inexact fuzzy-stochastic quadratic programming approach.
\newblock {\em Water resources management}, 27(2):451--473, 2013.

\bibitem{wang2003directional}
Guixiang Wang and Congxin Wu.
\newblock Directional derivatives and subdifferential of convex fuzzy mappings
  and application in convex fuzzy programming.
\newblock {\em Fuzzy Sets and Systems}, 138(3):559--591, 2003.

\bibitem{wu2004evaluate}
Hsien-Chung Wu.
\newblock Evaluate fuzzy optimization problems based on biobjective programming
  problems.
\newblock {\em Computers \& Mathematics with Applications}, 47(6-7):893--902,
  2004.

\bibitem{zadeh1996fuzzy}
Lotfi~A Zadeh.
\newblock Fuzzy sets.
\newblock In {\em Fuzzy sets, fuzzy logic, and fuzzy systems: selected papers
  by Lotfi A Zadeh}, pages 394--432. World Scientific, 1996.

\bibitem{zimmermann1978fuzzy}
H-J Zimmermann.
\newblock Fuzzy programming and linear programming with several objective
  functions.
\newblock {\em Fuzzy sets and systems}, 1(1):45--55, 1978.

\end{thebibliography}

\end{document}